\documentclass[12pt]{amsart}
\usepackage{fullpage,url,mathrsfs,setspace,amssymb}

\usepackage{color}

\newcommand{\defi}[1]{\textsf{#1}} 

\newcommand{\F}{{\mathbb F}}

\newcommand{\PP}{{\mathbb P}}
\newcommand{\Q}{{\mathbb Q}}
\newcommand{\R}{{\mathbb R}}
\newcommand{\Z}{{\mathbb Z}}

\def\bbar#1{\setbox0=\hbox{$#1$}\dimen0=.2\ht0 \kern\dimen0 
\overline{\kern-\dimen0 #1}}
 
\newcommand{\kbar}{{\bbar{k}}}
\newcommand{\Xbar}{X_{\kbar}}

\newcommand{\sF}{{\mathscr F}}

\newcommand{\calE}{{\mathcal E}}

\newcommand{\calR}{{\mathcal R}}
\newcommand{\calS}{{\mathcal S}}

\DeclareMathOperator{\rk}{rk}

\DeclareMathOperator{\Pic}{Pic}

\DeclareMathOperator{\Proj}{Proj}

\newcommand{\hideqed}{\renewcommand{\qed}{}} 
\newcommand{\into}{\hookrightarrow}

\newcommand{\legendre}[2]{\displaystyle\genfrac(){}{0}{#1}{#2}}

\newenvironment{romanenum}{\hfill \begin{enumerate} }{\end{enumerate}}

\newtheorem{theorem}{Theorem}[section]
\newtheorem{lemma}[theorem]{Lemma}
\newtheorem{corollary}[theorem]{Corollary}
\newtheorem{proposition}[theorem]{Proposition}

\theoremstyle{definition}

\newtheorem{conjecture}[theorem]{Conjecture}
\newtheorem{example}[theorem]{Example}

\theoremstyle{remark}
\newtheorem{remark}[theorem]{Remark}

\definecolor{webcolor}{rgb}{0,0,1}
\usepackage[
        colorlinks,
        linkcolor=webcolor,  filecolor=webcolor,  citecolor=webcolor, 
        backref,
        pdfauthor={Anthony V\'arilly-Alvarado}, 
]{hyperref}
\usepackage[alphabetic,backrefs,lite]{amsrefs} 

\begin{document}

\title[short paper title]{Density of rational points on isotrivial rational elliptic surfaces}
\subjclass[2000]{Primary 11 G35; Secondary 14 G05, 11 G05}
\author{Anthony V\'arilly-Alvarado}
\address{Department of Mathematics, Rice University, MS 136, Houston, Texas 77005, USA.}
\email{varilly@rice.edu}
\urladdr{http://math.rice.edu/\~{}av15}

\begin{abstract}
For a large class of isotrivial rational elliptic surfaces (with section), we show that the set of rational points is dense for the Zariski topology, by carefully studying variations of root numbers among the fibers of these surfaces.  We also prove that  these surfaces satisfy a variant of weak-weak approximation. Our results are conditional on the finiteness of Tate-Shafarevich groups for elliptic curves over the field of rational numbers.
\end{abstract}

\maketitle

\section{Introduction}

\subsection{Del Pezzo surfaces of degree $1$: a sample result}
Let $X$ be a smooth projective geometrically integral surface over a number field $k$.  Fix an algebraic closure $\kbar$ of $k$ and assume that $X$ is geometrically rational, i.e., the base extension $\Xbar := X\times_k\kbar$ is birational to the projective plane $\PP_\kbar^2$. A well-known result of Iskovskikh guarantees that $X$ is $k$-birational to either a rational conic bundle or a del Pezzo surface~\cite{Iskovskikh}.  

Del Pezzo surfaces that are not geometrically isomorphic to $\PP^1_\kbar\times\PP^1_\kbar$ are classified by their \defi{degree} $d := K_X^2$, an integer in the range $1\leq d \leq 9$. Segre and Manin have shown that if $X$ is a del Pezzo surface with $d \geq 2$, and if $X$ contains a $k$-point not lying on any exceptional curve, then $X(k)$ is dense in the Zariski topology; moreover, $X$ is $k$-unirational in this case~\cite[Theorem 29.4]{Manin}. Surfaces $X$ with $d = 1$ come furnished with a  rational point that does not lie on any exceptional curve (the base point of the anticanonical linear system). Hence the question: is $X(k)$ dense for the Zariski topology? One of our goals in this paper is to shed some light on this question, in the case when $k = \Q$.

\begin{theorem}
\label{C:diagonal}
Let $A, B$ be nonzero integers, and let $X$ be the del Pezzo surface of degree 1 over $\Q$ given by
 \begin{equation}
 \label{eq: diagonal surfaces}
 w^2 = z^3 + Ax^6 + By^6
 \end{equation}
in $\PP_\Q(1,1,2,3)$. Assume that Tate-Shafarevich groups of elliptic curves over $\Q$ with $j$-invariant $0$ are finite. If $3A/B$ is not a rational square, or if $A$ and $B$ are relatively prime and $9 \nmid AB$, then the rational points of $X$ are Zariski dense.
\end{theorem}

See \S\ref{ss: Main Results} for statements of our most general results.

\begin{remark}
\label{R:restriction}\ 
\begin{romanenum}
\item Every del Pezzo surface of degree $1$ is isomorphic to a \emph{smooth} sextic hypersurface in $\PP_k(1,1,2,3)$; conversely, a smooth sextic hypersurface in this weighted projective space is a del Pezzo surface of degree $1$~\cite{Kollar1996}*{Theorem III.3.5}.
\item Using explicit rational base changes, Ulas shows in~\cite{Ulas2}*{Corollary 4.4} that the conclusion of Theorem~\ref{C:diagonal} holds \emph{unconditionally} in the case $A = 1$.
\item The restriction in~\eqref{eq: diagonal surfaces}  that $A$ and $B$ are integers is not severe. If $A$ and $B$ are rational numbers, then one can clear denominators and rescale the variables to obtain an equation of the form~\eqref{eq: diagonal surfaces}.
\item Using the methods in~\cite{Varilly-Alvarado}, we may compute $\Pic X$ for the surfaces~\eqref{eq: diagonal surfaces}. If $\rk\Pic X = 1$, then $X$ is $\Q$-minimal, and is thus a ``genuine'' del Pezzo surface of degree $1$, i.e., $X$ is not the blow-up of a higher degree surface at closed $\Q$-points. This is the case, for example, if $A = B = p^3$, where $p > 3$ is a prime number; see~\cite{Varilly-Alvarado}*{Theorem~1.1}.
\end{romanenum}
\end{remark}

Blowing up the base point of the anticanonical linear system of a del Pezzo surface of degree $1$, we obtain a rational elliptic surface.  These are the main objects of study in our paper. However,
we state our results in \S\ref{ss: Main Results} in terms of hypersurfaces in $\PP_\Q(1,1,2,3)$ to emphasize the connection with del Pezzo surfaces of degree $1$. 

\subsection{Rational elliptic surfaces}
\label{ss: RatlEllSurfs}

Let $k$ be a number field, and let $(\calE,\rho,\sigma)$ be an \defi{elliptic surface with base $\PP^1_k$}, i.e., a smooth surface $\calE$ together with a morphism $\rho\colon\calE\to \PP^1_k$ that has a section $\sigma\colon\PP^1_k \to \calE$, such that $\rho$ is a relatively minimal elliptic fibration and has at least one (geometric) singular fiber. We often write $\calE$ instead of $(\calE,\rho,\sigma)$, the morphisms $\rho$ and $\sigma$ being understood. Suppose that $\calE\times_k\kbar$ is birational to $\PP^2_\kbar$ (in which case we say that $\calE$ is \defi{rational}). Then the generic fiber of $\rho$ is an elliptic curve $E/k(T)$ that can given by a Weierstrass equation of the form
\begin{equation}
\label{eq:weierstrass}
Y^2 = X^3 + a(T)X + b(T), \quad a(T), b(T) \in k[T],
\end{equation}
where 
\[
\deg a(T) \leq 4, \quad \deg b(T) \leq 6\quad\text{and}\quad \Delta := 4a(T)^3 + 27b(T)^2 \notin k.
\]
Conversely, any elliptic curve $E/k(T)$ of the form~\eqref{eq:weierstrass} uniquely extends to a rational elliptic surface with base $\PP^1_k$ (the Kodaira-N\'eron model of $E$).

We associate to $\calE$ a sextic hypersurface $X$ in the weighted projective space $\PP_k(1,1,2,3)$ as follows. Let $k[x,y,z,w]$ be the graded ring where the variables $x,y,z,w$ have weights $1,1,2,3$, respectively. Set $\PP_k(1,1,2,3) := \Proj k[x,y,z,w]$, and let $X$ be the sextic hypersurface
\begin{equation}
\label{eq: associated sextic}
w^2 = z^3 + G(x,y)z + F(x,y),
\end{equation}
where 
\[
G(x,y) = y^4a(x/y)\quad\text{and}\quad F(x,y) = y^6b(x/y). 
\]
The schemes $X$ and $\calE$ are birational: $X$ can be obtained from $\calE$ by contracting the image of the section $\sigma$ as well as the components of the singular fibers of $\rho$ that do not meet $\sigma(\PP^1_k)$. In general, $X$ will be a singular hypersurface.

We are interested in the qualitative distribution of the set $\calE(k)$.  In particular, we want to determine if the set $\calE(k)$ (equivalently, the set $X(k)$) is dense for the Zariski topology. Our investigations rely heavily on the root numbers of the fibers of $\rho$, and for this reason we focus our attention on the case $k = \Q$. 

To prove that $\calE(\Q)$ is Zariski dense, it suffices to show that for infinitely many $t \in \PP^1(\Q)$, the fiber $\calE_t$ of $\rho$ is an elliptic curve with positive Mordell-Weil rank. Assuming finiteness of Tate-Shafarevich groups, Nekov\'a\v r, Dokchitser and Dokchitser have shown that the root number of an elliptic curve $E/\Q$ is $(-1)^{\text{rank}(E)}$ (the parity conjecture; see~\cites{Nekovar, Dokchitsers}). We study the variation of root numbers among the smooth fibers of $\calE$, hoping to find infinitely many fibers with negative root number.

Rohrlich pioneered the study of variations of root numbers on algebraic families of elliptic curves in~\cite{Rohrlich}. Many authors followed suit; see, for example,~\cites{Manduchi,GM,GM2,Rizzo,CCH}. Some authors (notably~\cite{CCH}*{p.\ 686}) have observed that if the fibers of an elliptic surface lack ``geometric variation,'' then often there are simple formulae that describe the root numbers of these fibers; see, for example \cite{Rohrlich2}*{Corollary to Proposition~10}.  For this reason, we restrict our attention to \defi{isotrivial} rational elliptic surfaces, i.e., surfaces $\calE$ as above for which the modular invariant $j(E)$ has no $T$-dependence.  Such surfaces arise as families of (quadratic, cubic, quartic or sextic) twists of a fixed elliptic curve $E_0/\Q$:
\begin{romanenum}
\item (quadratic twists) $Y^2 = X^3 + af(T)^2X + bf(T)^3$ with $a, b \in k$, $f(T) \in k[T]$ and $1 \leq \deg f(T) \leq 2$, 
\item (cubic twists) $Y^2 = X^3 + f(T)^2$ with $f(T) \in k[T]$ and $1 \leq \deg f(T) \leq 3$, 
\item (quartic twists) $Y^2 = X^3 + f(T)X$ with $f(T) \in k[T]$ and $1 \leq \deg f(T) \leq 4$, 
\item (sextic twists) $Y^2 = X^3 + f(T)$ with $f(T) \in k[T]\setminus k[T]^2$ and $1 \leq \deg f(T) \leq 6$.
\end{romanenum} 
We use Rohrlich's formulae for local root numbers, together with those of Halberstadt and Rizzo~\cites{Halberstadt, Rizzo} to assemble root number formulae for quartic and sextic twists of elliptic curves over $\Q$ (see Propositions~\ref{T: root numbers of quartics} and~\ref{T: root numbers of sextics}, respectively). We then combine our explicit formulae for root numbers with an adaptation of a sieve introduced by Gouv\^ea, Mazur, and Greaves~\cites{GouveaMazur, Greaves}. The modified sieve allows us  to search for infinitely many pairs of fibers on a surface that have \emph{opposite} root numbers, which yields our density results (Theorems~\ref{T: Main Theorem} and~\ref{T: Main Theorem II}). For a similarly motivated idea, see~\cite{Manduchi}.


\subsection*{Outline of the paper}

In \S\ref{ss: Main Results}, we state our density theorems (Theorems~\ref{T: Main Theorem},~\ref{T: Main Theorem II} and~\ref{Thm: Weak-weak Approximation}), and we relate them to the literature, where many similar results can be found under the umbrella of Mazur's conjecture on the topology of rational points. In \S\ref{S:blow-up}, we make precise the relation between isotrivial rational elliptic surfaces and del Pezzo surfaces of degree $1$. In \S\ref{S:root numbers}, we present formulae for the root numbers of elliptic curves $E_\alpha/\Q$ of the form $y^2 = x^3 + \alpha$ or $y^2 = x^3 + \alpha x$, where $\alpha$ is a nonzero integer. We use our formulae to give conditions on integers $\alpha$ and $\beta$ under which $E_\alpha$ and $E_\beta$ have opposite root numbers (Corollaries~\ref{cor:flipping} and~\ref{cor:flippingII}), a crucial input in the proof of our density results.  In \S\ref{S:sieve}, we turn our attention to sieving, and present our modification of the squarefree sieve of Gouv\^ea, Mazur and Greaves.  In \S\ref{S:main thm}, we use this sieve to locate infinite families of fibers on elliptic surfaces with \emph{opposite} root number, and thus prove Theorems~\ref{T: Main Theorem} and~\ref{T: Main Theorem II}. In \S\ref{S:diagonal surfaces}, we specialize to the case of ``diagonal'' del Pezzo surfaces of degree $1$ over $\Q$. Finally, we prove Theorem~\ref{Thm: Weak-weak Approximation} in \S\ref{s: WWA}.

\subsection*{Acknowledgements}
I thank my thesis advisor, Bjorn Poonen, for numerous helpful conversations and for suggestions following a careful reading of the manuscript.  I thank David Zywina for useful conversations on sieves.  I thank Brendan Hassett and Cecilia Salgado for helpful conversations on elliptic surfaces. I thank Jean-Louis Colliot-Th\'el\`ene for suggesting that I use the methods in this paper to prove Theorem~\ref{Thm: Weak-weak Approximation}. Finally, I thank the anonymous referee for several useful suggestions.

\section{Main results}
\label{ss: Main Results}

Let $F(x,y) \in \Z[x,y]$ be a homogeneous binary form. We say that $F$  \defi{has a fixed prime divisor} if there is a prime number $p$ such that $F(a,b) \in p\Z$ for all $a, b \in \Z$. Note that if the content of $F(x,y)$ is not divisible by $p$, then $F(x,y) \bmod p$ has at most $\deg F(x,y)$ zeroes in $\PP^1(\F_p)$. Hence, if $p$ is a fixed prime divisor of $F(x,y)$, then $p + 1 \leq \deg F(x,y)$.

\subsection{Sextic twists and del Pezzo surfaces of degree \texorpdfstring{$1$}{1}}

Let $\rho\colon\calE\to \PP^1_\Q$ be an isotrivial rational elliptic surface whose associated sextic hypersurface $X \subseteq \PP_\Q(1,1,2,3)$ is smooth (hence a del Pezzo surface of degree $1$). We show in~\S\ref{S:blow-up} that $X$ must be isomorphic to a sextic of the form
\[
w^2 = z^3 + F(x,y),
\]
where $F(x,y)$ is a squarefree homogeneous form of degree $6$.  The generic fiber $E/\Q(T)$ of $\calE$ is isomorphic to 
\[
Y^2 = X^3 + b(T),\quad \text{where } b(T) = F(T,1)\text{ or }F(1,T),
\]
and can be thought of family of sextic twists.  We prove the following density result for this class of surfaces.

\begin{theorem}
\label{T: Main Theorem}
Let $F(x,y) \in \Z[x,y]$ be a homogeneous binary form of degree $6$; assume that the coefficients of $x^6$ and $y^6$ are nonzero. Let $X$ be the del Pezzo surface of degree $1$ over $\Q$ given by
\begin{equation}
\label{eq:sextic twists}
w^2 = z^3 + F(x,y)
\end{equation}
in $\PP_\Q(1,1,2,3)$.  Let $c$ be the content of $F$ and write $F(x,y) = cF_1(x,y)$ for some $F_1(x,y) \in \Z[x,y]$. Suppose that $F_1$ has no fixed prime divisors and that $F_1 = \prod_i f_i$, where the $f_i \in \Z[x,y]$ are irreducible homogeneous forms. Assume further that
\begin{equation}
\label{E:GaloisCondition}
\mu_3 \nsubseteq \Q[t]/f_i(t,1)\qquad\text{for some } i,
\end{equation}
where $\mu_3$ is the group of third roots of unity. Finally, assume that Tate-Shafarevich groups of elliptic curves over $\Q$ with $j$-invariant $0$ are finite. Then the rational points of $X$ are dense for the Zariski topology.
\end{theorem}

\begin{remark}
The restriction that $F(x,y) \in \Z[x,y]$ in Theorem~\ref{T: Main Theorem} is not severe; see Remark~\ref{R:restriction}(i).  Also, the assumption that the coefficients of $x^6$ and $y^6$ are nonzero is not a restriction: it can be achieved with a suitable linear transformation, without so changing the isomorphism class of $X$.
\end{remark}

We use Theorem~\ref{T: Main Theorem} to deduce Theorem~\ref{C:diagonal}, which addresses the question of Zariski density of rational points for ``diagonal'' del Pezzo surfaces of degree $1$ over $\Q$. We believe that the extraneous-looking hypotheses in Theorem~\ref{C:diagonal}, such as ``$3A/B$ is not a rational square'' or ``$9 \nmid AB$,'' are not necessary. Our method of proof, however, breaks down without them. For example, if $(A,B) = (27,16)$ then \emph{all} the nonsingular fibers of the corresponding elliptic surface $\rho\colon\calE\to\PP^1_\Q$ have \emph{positive} root number, and thus (conjecturally) even rank. In this particular example one can even show that all but finitely many fibers have rank at least $2$, whence Zariski density of rational points on $X$ is still true.  However, if, for example, $(A,B) = (243,16)$, then again all associated root numbers are positive, but we are unable to show rational points on $X$ are Zariski dense (see Example~\ref{E: example with sections} and Remark~\ref{R: no sections}).

\subsection{Quartic twists and (mildly singular) del Pezzo surfaces of degree \texorpdfstring{$1$}{1}}

Let $\rho\colon\calE\to \PP^1_\Q$ be an isotrivial rational elliptic surface and suppose that its generic fiber is of the form
\[
Y^2 = X^3 + a(T)X, \quad a(T) \in \Q[t],\ \deg a(T) \leq 4,
\]
which can be thought of as a family of quartic twists over $\Q$.  The associated hypersurface $X  \subseteq \PP_\Q(1,1,2,3)$, given by
\[
w^2 = z^3 +G(x,y)z,\quad G(x,y) := y^4a(x/y),
\]
is not smooth (and hence not a del Pezzo surface of degree $1$).  However, $X$ is not too far from being smooth: for example, when $G$ is squarefree, its singular locus consists of four $A_2$-singularities ($w = z = G(x,y) = 0$).  We prove the following density result for this class of surfaces.

\begin{theorem}
\label{T: Main Theorem II}
Let $G[x,y] \in \Z[x,y]$ be a squarefree homogeneous binary form of degree $4$; assume that the coefficients of $x^4$ and $y^4$ are nonzero.  Let $X$ be the hypersurface given by
\begin{equation}
\label{eq:quartic twists}
w^2 = z^3 + G(x,y)z
\end{equation}
in $\PP_\Q(1,1,2,3)$. Let $c$ be the content of $G$ and write $G(x,y) = cG_1(x,y)$ for some $G_1(x,y) \in \Z[x,y]$. Suppose that $G_1$ has no fixed prime divisors and that $G_1 = \prod_i g_i$, where the $g_i \in \Z[x,y]$ are irreducible homogeneous forms. Assume further that
\begin{equation}
\label{E:GaloisConditionII}
\mu_4 \nsubseteq \Q[t]/g_i(t,1)\qquad\text{for some } i,
\end{equation}
where $\mu_4$ is the group of fourth roots of unity.
Finally, assume that Tate-Shafarevich groups of elliptic curves over $\Q$ with $j$-invariant $1728$ are finite.
Then the rational points of $X$ are dense for the Zariski topology.
\end{theorem}

\begin{remark}
The assumption that the coefficients of $x^4$ and $y^4$ are nonzero is not a restriction: it can be achieved with a suitable linear transformation, without so changing the isomorphism class of $X$.
\end{remark}

\begin{remark}
\label{rem:Ulas}
In \cites{Ulas2,Ulas}, Ulas studies the question of Zariski density of rational points on certain del Pezzo surfaces of degree $1$ over $\Q$ by looking at explicit rational base-changes of their associated elliptic surfaces. His results do not depend on arithmetic conjectures and are thus stronger than ours, whenever there is an overlap---compare Theorem~\ref{T: Main Theorem} with~\cite{Ulas2}*{Theorems 2.1 and 2.2}, as well as Theorem~\ref{T: Main Theorem II} with~\cite{Ulas2}*{Theorems 3.1 and 3.2}.

\end{remark}

\subsection{Toward weak-weak approximation}

Write $\Omega_k$ for the set of places of a number field $k$, and let $k_v$ be the completion of $k$ at $v \in \Omega_k$.
Recall that a geometrically integral variety $X$ over $k$ satisfies \defi{weak-weak approximation}  if there exists a finite set $T \subseteq \Omega_k$ such that for every other finite set $S \subseteq \Omega_k$ with $S\cap T = \emptyset$, the image of the embedding
\[
X(k) \into \prod_{v \in S} X(k_v)
\]
is dense for the product topology of the $v$-adic topologies. We say that $X$ satisfies \defi{weak approximation} if we can take $T = \emptyset$. 

It is known that del Pezzo surfaces of low degree need not satisfy weak approximation; see \cite{CTSSD}*{Example~15.5}, \cite{SD},\cite{KreschTschinkelInt}*{Example~2} ,\cite{Varilly-Alvarado}*{Theorem~1.1} for counter-examples in degrees $4, 3, 2$ and $1$, respectively.  It is believed, however, that these surfaces satisfy weak-weak approximation.  More generally, a conjecture of Colliot-Th\'el\`ene predicts that unirational varietes satisfy weak-weak approximation (the conjecture implies a positive solution to the inverse Galois problem over number fields); see~{\cite{Serre}*{p.\ 30}}. Following a suggestion of Colliot-Th\'el\`ene, we use our modified squarefree sieve to show that the surfaces of Theorems~\ref{T: Main Theorem} and~\ref{T: Main Theorem II} satisfy a ``surrogate'' property that would be easily implied by weak-weak approximation.  For analogous results in this direction on certain elliptic surfaces without section, see~\cites{CTSkoSD}, and for more general fibrations over the projective line, see~\cite{CTSKoSDCrelle}.

\begin{theorem}
\label{Thm: Weak-weak Approximation}
Let $\rho\colon \calE \to \PP^1_\Q$ be an elliptic surface associated to one of the hypersurfaces considered in either Theorem~\ref{T: Main Theorem} or~\ref{T: Main Theorem II}. Let $\calR$ be the set of points $x \in \PP^1(\Q)$ such that the fiber $\calE_x = \rho^{-1}(x)$ is an elliptic curve of positive Mordell-Weil rank. Assume that Tate-Shafarevich groups of elliptic curves over $\Q$ with $j$-invariant $0$ or $1728$ are finite. Then there exists a finite set of primes $P_0$, containing the infinite prime, such that for every finite set of primes $P$ with $P\cap P_0 = \emptyset$, the image of the embedding 
\[
\calR \into \prod_{p \in P} \PP^1(\Q_p)
\]
is dense for the product topology of the $p$-adic topologies. 
\end{theorem}

\begin{remark}
The set $P_0$ in Theorem~\ref{Thm: Weak-weak Approximation} is effectively computed in the proof of the theorem.
\end{remark}

\subsection{Mazur's conjecture and related work}

In~\cite{Mazur}, Mazur made a series of conjectures on the topology of rational points on varieties, including the following.

\begin{conjecture}[\cite{Mazur}*{Conjecture~4}]
\label{Conj: Mazur}
Let $\calE \to \PP^1_\Q$ be an elliptic surface with base $\PP^1_\Q$. Then one of the following two conditions hold:
\begin{enumerate}
\item for all but finitely many $t \in \PP^1(\Q)$, the fiber $\calE_t$ is an elliptic curve with Mordell-Weil rank equal to zero,
\item the set of $t \in \PP^1(\Q)$ such that $\calE_t$ is an elliptic curve with positive Mordell-Weil rank is dense in $\PP^1(\R)$.
\end{enumerate}
\end{conjecture}

Many authors have shown since that (2) holds for a range of elliptic surfaces.  In particular, the set $\calE(\Q)$ is dense in the Zariski topology for these surfaces. For example, in~\cite{Rohrlich}*{Theorem~3} Rohrlich shows, unconditionally and using elementary methods, that if $f(t) \in \Q[t]$ is a quadratic polynomial, then the Kodaira-N\'eron model $\calE$ of the elliptic curve over $\Q(T)$ given by
\[
Y^2 = X^3 + af(T)^2X + bf(T)^3\qquad a, b \in \Q
\]
satisfies part (2) of Conjecture~\ref{Conj: Mazur}, provided that there exists $t \in \Q$ such that $f(t) \neq 0$ and that $\calE_t$ has positive Mordell-Weil rank. Munshi has recently extended this result to rational elliptic surfaces over real number fields, provided there are at least two fibers of positive rank and one fiber with a 2-torsion point defined over the ground field~\cite{MunshiIJNT}*{Theorem~2}.

Kuwata and Wang have a similar result to Rohrlich's for quadratic twists by cubic polynomials \cite{KW}. The resulting isotrivial elliptic surfaces, however, are not rational;  they are $K3$ surfaces.  In~\cite{MunshiNT}, Munshi examines Conjecture~\ref{Conj: Mazur} for many kinds of isotrivial rational elliptic surfaces, including cubic twists, by studying ``horizontal'' elliptic or conic bundle structures on these surfaces.  There is surprisingly little overlap between Munshi's and our investigations; in fact, our methods cannot yield density results for cubic twists (the squarefreeness of $F(x,y)$ in \eqref{eq:sextic twists} is central to our sieving argument). We have conditionally addressed the question of Zariski density of rational points on some of the isotrivial cases left open by Munshi~\cite{MunshiNT}*{\S7}. 

Assuming the parity conjecture, Manduchi has shown that conclusion (2) of Conjecture~\ref{Conj: Mazur} holds for large families of \emph{non-isotrivial} elliptic surfaces with base $\PP^1_\Q$; see~\cite{Manduchi}.  Over a general number field, and assuming the Birch--Swinnerton-Dyer conjecture, as well as a conjecture of Deligne and Gross, Grant and Manduchi have shown that rational points are \emph{potentially dense} for non-isotrivial elliptic surfaces over a rational or elliptic base; see~\cites{GM,GM2}.  Ulas has obtained density results on extensive families of rational non-isotrivial elliptic surfaces by studying explicit rational base changes~\cite{Ulas2}*{Theorems 5.1 and 5.3} (see Remark~\ref{rem:Ulas} as well). Helfgott has also obtained density results for elliptic surfaces through his study of average root numbers in families~\cite{Helfgott}. His results depend on classical arithmetical conjectures.

Elkies has suggested that Conjecture~\ref{Conj: Mazur} is false; he has a heuristic which indicates that certain families of quadratic twists by a polynomial of high degree should yield counterexamples~\cite{Elkies}. \\

In~\cite{CTSkoSD}, Colliot-Th\'el\`ene, Swinnerton-Dyer and Skorobogatov study the vertical Brauer-Manin obstruction of a large class of elliptic surfaces \emph{without section}. In particular, they show that the set of rational points of the elliptic surfaces they study is dense for the Zariski topology as soon as it is non-empty.  Their  results are conditional on the finiteness of Tate-Shafarevich groups and Schinzel's hypothesis (a wild generalization of the twin primes conjecture).

\section{Isotrivial elliptic surfaces and del Pezzo surfaces of degree \texorpdfstring{$1$}{1}}
\label{S:blow-up}

Let $k$ be a number field, and let $(\calE,\rho,\sigma)$ be an isotrivial rational elliptic surface with base $\PP^1_k$. The generic fiber $E/k(T)$ of $\calE$ is isomorphic to a curve in the list (i)--(iv) of \S\ref{ss: RatlEllSurfs}.  Suppose that the sextic hypersurface $X \subseteq \PP_k(1,1,2,3)$ associated to $\calE$ is smooth (and hence a del Pezzo surface of degree $1$). Then a straightforward (albeit tedious) application of the Jacobian criterion shows that $E/k(T)$ must be a family of sextic twists (iv), with $f(T)$ squarefree.  Alternatively, we may argue as follows.  Since $X_\kbar$ is isomorphic to $\PP^2_\kbar$ blown-up at 9 distinct points in general position \cite{Manin}*{}, it follows from~\cite{Shioda}*{Theorem~10.11} that the Mordell-Weil lattice of $E_{\kbar(T)}$ has rank $8$. From the Shioda-Tate formula~\cite{Shioda}*{Corollary~5.3}, we deduce that $\calE_\kbar$ has no reducible fibers, i.e., the singular fibers of $\rho_\kbar\colon \calE_\kbar \to \PP^1_\kbar$ must be of type $\textup{I}_0$ or $\textup{II}$, in Kodaira's notation.  The isotriviality of $\calE$ precludes singular fibers of type $\textup{I}_0$ (because these fibers are semi-stable).  Looking at Persson's classification~\cite{Persson} of rational elliptic surfaces, we conclude that $\calE_\kbar$ must have six singular fibers of type $\textup{II}$.  A quick application of Tate's algorithm to the Kodaira-N\'eron models of the possible generic fibers (i)--(iv) in \S\ref{ss: RatlEllSurfs} leaves (iv) as the only possibility, under the additional hypothesis that $f(T)$ is squarefree. We have thus shown the following proposition.

\begin{proposition}
Let $k$ be a number field and let $(\calE,\rho,\sigma)$ be an isotrivial rational elliptic surface with base $\PP^1_k$. Suppose that the sextic hypersurface $X \subseteq \PP^1_k(1,1,2,3)$ associated to $\calE$ is smooth.  Then $X$ is isomorphic to a hypersurface of the form
\[
w^2 = z^3 + F(x,y),
\]
where $F(x,y)$ is a squarefree homogeneous form.
\qed
\end{proposition}

\section{Root numbers and flipping}
\label{S:root numbers}

Let $E$ be an elliptic curve over $\Q$. The \defi{root number} $W(E)$ of $E$ is defined as a product of local factors
\[
W(E) = \prod_{p \leq\infty} W_p(E),
\]
where $p$ runs over the rational prime numbers and infinity, $W_p(E) \in \{\pm 1\}$ and $W_p(E) = +1$ for all but finitely many $p$. The \defi{local root number} $W_p(E)$ of $E$ at $p$ is defined in terms of epsilon factors of Weil-Deligne representations of $\Q_p$; it is an invariant of the isomorphism class of the base extension $E_{\Q_p}$ of $E$. For a definition of these local factors see~\cites{Deligne,Tate}.  If $p$ is a prime of good reduction for $E$ then $W_p(E) = +1$; furthermore, $W_\infty(E) = -1$ (see~\cite{Rohrlich}). The computation of $W_p(E)$ for primes of bad reduction in terms of data associated to a Weierstrass model of $E$ has been studied by various authors, particularly by Rohrlich, Halberstadt and Rizzo~\cites{Rohrlich,Halberstadt,Rizzo}. In this section, we build on their work to give formulae for the root numbers of elliptic curves over $\Q$ of the form
\[
y^2 = x^3 + \alpha \quad\text{and}\quad y^2 = x^3 + \alpha x \qquad (\alpha \neq 0).
\]
Our formula for the root number of $y^2 = x^3 + \alpha$ has a flavor different from that found in~\cite{Liverance}; in particular, it is visibly insensitive to primes $p \geq 5$ whose \emph{square} does not divide $\alpha$. 

Conjecturally, the root number $W(E)$ of an elliptic curve is the sign in the functional equation for the $L$-series $L(E,s)$ of $E$:
\[
(2\pi)^{-s}\Gamma(s) N^{s/2} L(E,s) = W(E) (2\pi)^{2-s}\Gamma(2-s) N^{(2-s)/2} L(E,2-s),
\]
where $N$ is the conductor of $E$, and $\Gamma(s)$ is the usual Gamma function. According to the Birch--Swinnerton-Dyer conjecture,
\begin{equation}
\label{eq:parity}
W(E) = (-1)^{\text{rank}(E)}.
\end{equation}
Equality~\eqref{eq:parity} is itself known as the \defi{parity conjecture}.  By work of Nekov\'a\v{r}, Dokchitser and Dokchitser, the finiteness of Tate-Shafarevich groups is enough to prove the parity conjecture~\cites{Nekovar,Dokchitsers}.  

\subsection*{Notation}
\label{ss:notation}
In addition to the notation introduced above, we use the following conventions.
Throughout, for a prime $p \in \Z$ we denote the corresponding $p$-adic valuation by $v_p$. If $a$ is a nonzero integer then $\legendre{a}{p}$ will denote the usual Legendre symbol; if $m$ is an odd positive integer then $\legendre{a}{m}$ will denote the usual Jacobi symbol.

\subsection{The root number of \texorpdfstring{$E_\alpha : y^2 = x^3 + \alpha$}{}}
\label{ss: j = 0}

Let $\alpha$ be a nonzero integer.  We give a closed formula for the root number of the elliptic curve $E_\alpha/\Q : y^2 = x^3 + \alpha$, in terms of $\alpha$.  Throughout, we write $W(\alpha)$ for this root number and $W_p(\alpha)$ for the local root number of $E_\alpha$ at $p$. We begin by determining $W_2(\alpha)$ and $W_3(\alpha)$.

\begin{lemma}
\label{L:local root numbers at 2 and 3}
Let $\alpha$ be a nonzero integer. Define $\alpha_2$ and $\alpha_3$ by $\alpha = 2^{v_2(\alpha)}\alpha_2 = 3^{v_3(\alpha)}\alpha_3$. 
Then
\begin{align*}
W_2(\alpha) &= 
\begin{cases}
{-1} & \text{if $v_2(\alpha) \equiv 0$ or $2 \bmod 6$}; \\
& \text{or if $v_2(\alpha) \equiv 1, 3, 4$ or $5 \bmod 6$ and $\alpha_2 \equiv 3 \bmod{4}$; }\\
+1 & \text{otherwise,}
\end{cases} \\
W_3(\alpha) &= 
\begin{cases}
{-1} & \text{if $v_3(\alpha) \equiv 1$ or $2 \bmod 6$ and $\alpha_3 \equiv 1 \bmod 3$; }\\
     & \text{or if $v_3(\alpha) \equiv 4$ or $5\bmod 6$ and $\alpha_3 \equiv 2 \bmod 3$; }\\
     & \text{or if $v_3(\alpha) \equiv 0 \bmod 6$ and $\alpha_3 \equiv 5$ or $7 \bmod 9$; }\\
     & \text{or if $v_3(\alpha) \equiv 3 \bmod 6$ and $\alpha_3 \equiv 2$ or $4 \bmod 9$,} \\
+1   & \text{otherwise.}
\end{cases}
\end{align*}
\end{lemma}

\begin{proof}
According to~\cite{Rizzo}*{\S1.1}, to determine the local root number at $p$ of an elliptic curve given in Weierstrass form, we must find the smallest vector with nonnegative entries 
\begin{equation}
\label{eq:smallest vector}
(a,b,c) := (v_p(c_4),v_p(c_6),v_p(\Delta)) + k(4,6,12)
\end{equation}
for $k \in \Z$, where $c_4, c_6$ and $\Delta$ are the usual quantities associated to a Weierstrass equation (see~\cite{Silverman}*{Ch.\ III}). For the curves in question we have
\[
c_4 = 0,\quad c_6 = -2^5\cdot 3^3\cdot \alpha,\quad\text{and}\quad \Delta = -2^4\cdot 3^3 \cdot \alpha^2,
\]
whence
\[
(v_p(c_4),v_p(c_6),v_p(\Delta)) = (\infty,v_p(\alpha),2v_p(\alpha)) + 
\begin{cases}
(0,5,4) & \text{if } p = 2, \\
(0,3,3) & \text{if } p = 3, 
\end{cases}
\]
Now it is a simple matter of using the tables in~\cite{Rizzo}*{\S1.1} to compute local root numbers. We illustrate the computation of $W_2(\alpha)$ in one example. Suppose that $v_2(\alpha) \equiv 4 \bmod 6$. Then $(a,b,c) = (\infty,3,0)$, and according to the entries under $(\geq 4,3,0)$ in Table~III of~\cite{Rizzo}, we have $W_2(\alpha) = -1$ if and only if $c_6' := c_6/2^{v_2(c_6)} \equiv 3 \bmod 4$, i.e., if and only if $\alpha_2 \equiv 3 \bmod 4$. All other local root number computations are similar and we omit the details.
\end{proof}

\begin{remark}
We take the opportunity to note that the entry $(\geq{}\!5,6,9)$ in Table~II of~\cite{Rizzo} has a typo. The ``special condition'' should read $c_6' \not\equiv \pm 4 \bmod 9$.
\end{remark}

The elliptic curve $E_\alpha$ has potential good reduction at every nonarchimedean place. We will use the following proposition, due to Rohrlich, which gives a formula for the local root numbers of an elliptic curve at primes $p \geq 5$ of potential good reduction.

\begin{proposition}[{\cite{Rohrlich}*{Proposition~2}}] 
\label{prop: Rohrlich}
Let $p \geq 5$ be a rational prime, and let $E/\Q_p$ be an elliptic curve with potential good reduction.  Write $\Delta \in \Q_p^*$ for the discriminant of any generalized Weierstrass equation for $E$ over $\Q_p$. Let
\[
e := \frac{12}{\gcd(v_p(\Delta),12)}.
\]
Then
\[
W_p(E) = 
\begin{cases}
1 & \text{if } e = 1, \\
\legendre{-1}{p} & \text{if $e = 2$ or $6$}, \\
\legendre{-3}{p} & \text{if $e = 3$}, \\
\legendre{-2}{p} & \text{if $e = 4$}. \\
\end{cases} \eqno\qed
\]
\end{proposition}

\begin{proposition}[Root numbers for $y^2 = x^3 + \alpha$]
\label{T: root numbers of sextics}
Let $\alpha$ be a nonzero integer, and let 
\begin{equation}
\label{E:prod}
R(\alpha) = W_2(\alpha)\legendre{-1}{\alpha_{2}}W_3(\alpha)(-1)^{v_3(\alpha)}.
\end{equation}
Then
\begin{equation}
\label{eq:root numbers}
W(\alpha) = -R(\alpha)\prod_{\substack{p^2 | \alpha \\ p\geq 5}}
\begin{cases}
1                & \text{if } v_p(\alpha) \equiv 0, 1,3,5\bmod 6, \\
\legendre{-3}{p} & \text{if } v_p(\alpha) \equiv 2,4\bmod 6.
\end{cases}
\end{equation}
Let $\beta$ be another nonzero integer, and suppose that $\alpha \equiv \beta \bmod 2^{v_2(\alpha) + 2}\cdot 3^{v_3(\alpha) + 2}$. Then $R(\alpha) = R(\beta)$.
\end{proposition}

\begin{proof}
Since $\Delta(E_\alpha) = -2^43^3\alpha^2$, applying Proposition~\ref{prop: Rohrlich} we obtain
\begin{equation}
W(\alpha) = - W_2(\alpha)W_3(\alpha)
\prod_{\substack{p\,|\,\alpha \\ p\geq 5}}
\begin{cases}
1                & \text{if } v_p(\alpha) \equiv 0 \bmod 6, \\
\legendre{-1}{p} & \text{if } v_p(\alpha) \equiv 1,3,5 \bmod 6, \\
\legendre{-3}{p} & \text{if } v_p(\alpha) \equiv 2,4\bmod 6.
\end{cases}
\label{E:roughroot}
\end{equation}
Let $r$ be the product of the primes $p \ge 5$ such that $v_p(\alpha)=1$, let $b = \alpha/r$ and set
$$
\alpha_{2} := \frac{\alpha}{2^{v_2(\alpha)}}, \qquad
b_{2} := \frac{b}{2^{v_2(b)}}.
$$
Note that $r = \alpha_{2}/b_{2} = \alpha/b$. We may rewrite~(\ref{E:roughroot}) as
\begin{equation}
W(\alpha) = - W_2(\alpha)W_3(\alpha)
\legendre{-1}{r}
\prod_{\substack{p \,|\, b \\ p\geq 5}}
\begin{cases}
1                & \text{if } v_p(\alpha) \equiv 0 \bmod 6, \\
\legendre{-1}{p} & \text{if } v_p(\alpha) \equiv 1, 3,5 \bmod 6, \\
\legendre{-3}{p} & \text{if } v_p(\alpha) \equiv 2,4\bmod 6.
\end{cases}
\label{E:betterroot}
\end{equation}
On the other hand, we have
\[
\legendre{-1}{r} =
\legendre{-1}{\alpha_{2}/b_{2}} =
\legendre{-1}{\alpha_{2}}\cdot \legendre{-1}{b_{2}} =
\legendre{-1}{\alpha_{2}}\cdot \legendre{-1}{3}^{v_3(\alpha)}\cdot
\prod_{\substack{p \,|\, b \\ p \geq 5}}\legendre{-1}{p}^{v_p(\alpha)},
\]
so we can write~(\ref{E:betterroot}) as
\begin{align*}
W(\alpha) &= -\left[W_2(\alpha)\legendre{-1}{\alpha_{2}}W_3(\alpha)(-1)^{v_3(\alpha)}\right]
\prod_{\substack{p \,|\, b \\ p\geq 5}}
\begin{cases}
\legendre{-1}{p}^{v_p(\alpha)}                      & \text{if } v_p(\alpha) \equiv 0 \bmod 6, \\
\legendre{-1}{p}^{1 + v_p(\alpha)}                  & \text{if } v_p(\alpha) \equiv 1,3,5 \bmod 6, \\
\legendre{-3}{p}\!\cdot\!\legendre{-1}{p}^{v_p(\alpha)} & \text{if } v_p(\alpha) \equiv 2,4\bmod 6,
\end{cases}
 \\
 &= -R(\alpha)\prod_{\substack{p^2 \,|\, \alpha \\ p\geq 5}}
\begin{cases}
1                & \text{if } v_p(\alpha) \equiv 0,1,3,5\bmod 6, \\
\legendre{-3}{p} & \text{if } v_p(\alpha) \equiv 2,4\bmod 6.
\end{cases}
\end{align*}
as desired, because $p \,|\, b,\ p\geq 5 \iff p^2 \,|\,\alpha,\ p\geq 5$. 

To prove the last claim of the proposition, note that if $\alpha \equiv \beta \bmod 2^{v_2(\alpha)+2}\cdot3^{v_3(\alpha)+2}$ then $v_2(\alpha) = v_2(\beta)$, $v_3(\alpha) = v_3(\beta)$, and we have
\[
\frac{\alpha}{2^{v_2(\alpha)}} \equiv \frac{\beta}{2^{v_2(\beta)}} \bmod 4\quad\text{and}\quad
\frac{\alpha}{3^{v_3(\alpha)}} \equiv \frac{\beta}{3^{v_3(\beta)}} \bmod 9.
\]
The claim now follows from Lemma~\ref{L:local root numbers at 2 and 3} 
\end{proof}

The following corollary describes conditions on two nonzero integers $\alpha$ and $\beta$ which guarantee that the elliptic curves $y^2 = x^3 + \alpha$ and $y^2 = x^3 + \beta$ have \emph{opposite} root numbers. This is one of the key inputs to the proof of Theorem~\ref{T: Main Theorem}. This corollary is similar in spirit to~\cite{Manduchi}*{Corollary~2.1}.

\begin{corollary}[Flipping I]
Let $\alpha, \beta$ be nonzero integers such that 
\begin{enumerate}
\item $\alpha \equiv \beta \bmod 2^{v_2(\alpha) + 2}\cdot 3^{v_3(\alpha) + 2}$,
\item $\alpha = c\ell$, where $\ell$ is squarefree and $\gcd(c,\ell) = 1$,
\item $\beta = c q^{2+ 6k} \eta$, where $\eta$ is square free, $\gcd(c,\eta) = \gcd(q,c\eta) = 1$, $k\geq 0$, $q \ge 5$ is prime and $q\equiv 2\bmod 3$.
\end{enumerate}
Then $W(\alpha) = -W(\beta)$.
\label{cor:flipping}
\end{corollary}

\begin{proof}
The first condition ensures that $R(\alpha) = R(\beta)$. Since $\ell$ is squarefree and $\gcd(c,\ell) = 1$, the only primes greater than $3$ contributing to $W(\alpha)$ are those whose square divides $c$. Similarly, since $\eta$ is squarefree and $\gcd(c,\eta) = \gcd(q,\eta) = 1$, the only primes  greater than $3$ contributing to $W(\beta)$ are those whose square divides $c$, and $q$. Since $\gcd(q,c) = 1$, $q \ge 5$ and $q \equiv 2\bmod 3$, we have
\[
W(\beta) = \legendre{-3}{q}W(\alpha) = -W(\alpha)\eqno \qed
\]
\hideqed
\end{proof}

\begin{remark}
\label{R: reduction}
To prove Zariski density of rational points on the elliptic surface $\calE \to \PP^1_\Q$ associated to a del Pezzo of degree $1$ as in Theorem~\ref{T: Main Theorem}, it is enough to do the following. First, prove that there exist infinite sets $\sF_1$ and $\sF_2$ of coprime pairs of integers such that whenever $(m_1,n_1) \in \sF_1$ and $(m_2,n_2) \in \sF_2$ then
\begin{enumerate}
\item $\alpha := F(m_1,n_1)$ and $\beta := F(m_2,n_2)$ are nonzero integers.
\item The integers $\alpha$ and $\beta$ satisfy the hypotheses of Corollary~\ref{cor:flipping}.
\end{enumerate}
Then, by Corollary~\ref{cor:flipping}, we know that either
\[
W(F(m,n)) = -1 \text{ for all } (m,n) \in \sF_1,
\]
or
\[
W(F(m,n)) = -1 \text{ for all } (m,n) \in \sF_2.
\]
Hence, there are infinitely many closed fibers of $\calE \to \PP^1_\Q$ with negative root number. Assuming the parity conjecture, this gives an infinite number of closed fibers with infinitely many points, and hence a Zariski dense set of rational points on $\calE$.
\end{remark}

\subsection{The root number of \texorpdfstring{$E_\alpha : y^2 = x^3 + \alpha x$}{}}

Next, we give a closed formula for the root number of the elliptic curve $E_\alpha/\Q : y^2 = x^3 + \alpha x$, in terms of the nonzero integer $\alpha$.  The proofs mirror those of \S\ref{ss: j = 0}, and thus we have omitted them. Throughout this section, we write $W(\alpha)$ for the root number of $E_\alpha$ and $W_p(\alpha)$ for the local root number at $p$ of $E_\alpha$. 

\begin{lemma}
\label{L:local root numbers at 2 and 3 II}
Let $\alpha$ be a nonzero integer. Define $\alpha_2$ and $\alpha_3$ by $\alpha = 2^{v_2(\alpha)}\alpha_2 = 3^{v_3(\alpha)}\alpha_3$. Then
\begin{align*}
W_2(\alpha) &= 
\begin{cases}
{-1} & \text{if $v_2(\alpha) \equiv 1$ or $3 \bmod 4$ and $\alpha_2 \equiv 1$ or $3 \bmod 8$; }\\
     & \text{or if $v_2(\alpha) \equiv 0 \bmod 4$ and $\alpha_2 \equiv 1, 5, 9, 11,13$ or $15 \bmod  
       16$; }\\
     & \text{or if $v_2(\alpha) \equiv 2 \bmod 4$ and $\alpha_2 \equiv 1, 3, 5, 7, 11$ or $15 \bmod 
       16$;} \\
+1   & \text{otherwise,}
\end{cases} \\
W_3(\alpha) &= 
\begin{cases}
{-1} & \text{if $v_3(\alpha) \equiv 2 \bmod 4$, }\\
+1   & \text{otherwise.}
\end{cases}
\end{align*}
\end{lemma}

\begin{proof}
Proceed as in the proof of Lemma~\ref{L:local root numbers at 2 and 3}, using the quantities
\[
c_4 = -2^4\cdot 3\cdot\alpha,\quad c_6 = 0,\quad\text{and}\quad \Delta = -2^6 \cdot \alpha^3. \eqno{\qed}
\]
\hideqed
\end{proof}

\begin{proposition}[Root numbers for $y^2 = x^3 + \alpha x$]
\label{T: root numbers of quartics}
Let $\alpha$ be a nonzero integer, and let 
\begin{equation}
\label{E:prodII}
R(\alpha) = W_2(\alpha)\legendre{-1}{\alpha_{2}}W_3(\alpha)(-1)^{v_3(\alpha)}.
\end{equation}
Then
\[
W(\alpha) = -R(\alpha)\prod_{\substack{p^2 | \alpha \\ p\geq 5}}
\begin{cases}
\legendre{-1}{p} & \text{if } v_p(\alpha) \equiv 2 \bmod 4, \\
\legendre{2}{p}  & \text{if } v_p(\alpha) \equiv 3 \bmod 4.
\end{cases}
\]
Let $\beta$ be another nonzero free integer, and suppose that $\alpha \equiv \beta \bmod 2^{v_2(\alpha) + 4}\cdot 3^{v_3(\alpha)}$. Then $R(\alpha) = R(\beta)$. 
\qed
\end{proposition}

The following corollary, which parallels Corollary~\ref{cor:flipping}, describes conditions on two nonzero integers $\alpha$ and $\beta$ that guarantee that the elliptic curves $y^2 = x^3 + \alpha x$ and $y^2 = x^3 + \beta x$ have \emph{opposite} root numbers. This is one of the key inputs to the proof of Theorem~\ref{T: Main Theorem II}.

\begin{corollary}[Flipping II]
Let $\alpha, \beta$ be nonzero integers such that 
\begin{enumerate}
\item $\alpha \equiv \beta \bmod 2^{v_2(\alpha) + 4}\cdot 3^{v_3(\alpha)}$,
\item $\alpha = c \ell$, where $\ell$ is squarefree and $\gcd(c,\ell) = 1$,
\item $\beta = c q^{2+ 4k} \eta$, where $\eta$ is square free, $\gcd(c,\eta) = \gcd(q,c \eta) = 1$, $k\geq 0$, $q \ge 5$ is prime and $q\equiv 3\bmod 4$; or $\beta = c p^{3 + 4k} \eta$, where $\eta$ is square free, $\gcd(c,\eta) = \gcd(q,c \eta) = 1$, $k\geq 0$, $q \ge 5$ is prime and $q\equiv 3$ or $5 \bmod 8$.
\end{enumerate}
Then $W(\alpha) = -W(\beta)$.
\qed
\label{cor:flippingII}
\end{corollary}

\begin{remark}
\label{R: reductionII}
To prove Zariski density of rational points on the elliptic surface $\calE \to \PP^1_\Q$ associated to a sextic hypersurface as in Theorem~\ref{T: Main Theorem II}, it is enough to do the following. First, prove that there exist infinite sets $\sF_1$ and $\sF_2$ of coprime pairs of integers such that whenever $(m_1,n_1) \in \sF_1$ and $(m_2,n_2) \in \sF_2$ then
\begin{enumerate}
\item $\alpha := G(m_1,n_1)$ and $\beta := G(m_2,n_2)$ are nonzero integers.
\item The integers $\alpha$ and $\beta$ satisfy the hypotheses of Corollary~\ref{cor:flippingII}.
\end{enumerate}
Then, arguing as in Remark~\ref{R: reduction} (using Corollary~\ref{cor:flippingII}) we find infinitely many closed fibers of $\calE \to \PP^1_\Q$ with negative root number. This gives a Zariski dense set of rational point for $\calE$, assuming the parity conjecture.
\end{remark}

\section{The Modified Square-free Sieve}
\label{S:sieve}

In this section we present a variation of a squarefree sieve by Gouv\^ea, Mazur and Greaves \cites{GouveaMazur,Greaves}. It is the tool that allows us to identify families of fibers with negative root numbers on certain elliptic surfaces.

Let $F(m,n) \in \Z[m,n]$ be a binary homogeneous form of degree $d$, not divisible by the square of a nonunit in $\Z[m,n]$. Write $F = \prod_{i=1}^t f_i$, where the $f_{i}(m,n) \in \Z[m,n]$ are irreducible, and assume that $\deg f_i \leq 6$ for all $i$.  Applying a unimodular transformation we may (and do) assume that the coefficients of $m^d$ and $n^d$ in $F(m,n)$ are nonzero. Call their respective coefficients $a_d$ and $a_0$. Write $F(m,n) = a_d\prod(m - \theta_i n)$, where the $\theta_i$ are algebraic numbers and $1 \leq i \leq d$. Let 
\[
\Delta(F) = \bigg|a_0a_d^{2d-1}\prod_{i\neq j}(\theta_i - \theta_j)\bigg|;
\] 
this is essentially the discriminant of the form $F$.  It is nonzero if and only if $F$ contains no square factors.

Fix a positive integer $M$, as well as a subset $\calS$ of $(\Z/M\Z)^2$. Our goal is to count pairs of integers $(m,n)$ such that $(m\bmod M, n\bmod M) \in \calS$ and $F(m,n)$ is not divisible by $p^2$ for any prime number $p$ such that $p\,\nmid\, M$. This will allow us to give an asymptotic formula for the number of pairs of integers $(m,n)$ with $0\leq m,n \leq x$ such that
\[
F(m,n) = \nu\cdot \ell,
\]
where $\nu$ is a \emph{fixed} integer and $\ell$ is a squarefree integer such that $\gcd(\nu,\ell) = 1$. The case $\nu = 1$ is handled by Gouv\^ea and Mazur in~\cite{GouveaMazur} under the additional assumption that $\deg f_i \leq 3$, and extended by Greaves in~\cite{Greaves} to the case $\deg f_i \leq 6$. We build upon their work to prove an asymptotic formula when $\nu > 1$.

\begin{remark}
The role of the set $\calS$ above is to ``decouple'' the congruence conditions on $(m,n)$ from the sieving process.  This artifact, suggested to us by Bjorn Poonen after an initial reading of the manuscript, cleans up the analytic proofs in the main-term estimate for our sieve.
\end{remark}

We make use of the following (mild variation of an) arithmetic function studied by Gouv\^ea and Mazur: put $\rho(1) = 1$, and for $k \geq 2$ let
\[
\rho(k) = \#\{(m,n) \in \Z^2 : 0 \leq m,n \leq k-1, F(m,n) \equiv 0 \bmod k \}.
\]
By the Chinese remainder theorem, the function $\rho$ is multiplicative; i.e., if $k_1$ and $k_2$ are relatively prime positive integers then $\rho(k_1k_2) = \rho(k_1)\rho(k_2)$.

\begin{lemma}[{\cite{GouveaMazur}*{Lemma~3(2)}}]
\label{L:rho props}
For fixed $F$ as above and squarefree $\ell$, we have $\rho(\ell^2) = O(\ell^2\cdot d_k(\ell))$ as $\ell \to \infty$, where $k = \deg(F)+1$ and $d_k(\ell)$ denotes the number of ways in which $\ell$ can be expressed as a product of $k$ factors. In particular, $\rho(p^2) = O(p^2)$ as $p \to \infty$.
\qed
\end{lemma}

We can now state the main result of this section.

\begin{theorem}
\label{T: Sieve}
Let $F(m,n)\in \Z[m,n]$ be a homogeneous binary form of degree $d$. Assume that no square of a nonunit in $\Z[m,n]$ divides $F(m,n)$, and that no irreducible factor of $F$ has degree greater than $6$. Fix a positive integer $M$, as well as a subset $\calS$ of $(\Z/M\Z)^2$. Let $N(x)$ be the number of pairs of integers $(m,n)$ with $0 \leq m,n \leq x$ such that $(m \bmod M, n \bmod M) \in \calS$ and $F(m,n)$ is not divisible by $p^2$ for any prime $p$ such that $p\nmid M$. Then
\[
N(x) = Cx^2 + O\left(\frac{x^2}{(\log x)^{1/3}}\right)\qquad\textup{as }x \to \infty,
\]
where 
\[
C = \frac{|\calS|}{M^2}\prod_{p\,\nmid\,M}\left(1 - \frac{\rho(p^2)}{p^{4}}\right).
\]
\label{thm:ModifiedSieve}
\end{theorem}

\begin{remark}
By Lemma~\ref{L:rho props}, $\rho(p^2) = O(p^2)$ as $p \to \infty$ for a fixed $F$, so the infinite product defining $C$ converges.
\end{remark}

Heuristically, the condition that $F(m,n)$ be squarefree outside a prescribed integer is well approximated by the condition that $F(m,n)$ not be divisible by the square of a prime that is ``small relative to $x$.'' More precisely, let $\xi = (1/3)\log x$ and define the principal term
\begin{align*}
N'(x) = \{(m,n) \in \Z^2 : \ 0 \leq m,n\leq x,\ \ &F(m,n) \not\equiv 0 \bmod p^2 \text{ for all } p \leq
                                                   \xi, p\,\nmid\,M \\
						                             & \text{and } (m\bmod M, n\bmod M) \in \calS \}.
\end{align*}
Let $F = \prod_{i=1}^t f_i$ be a factorization of $F$ into irreducible binary forms. Define the partial $i$-th error term $E_i(x)$ as follows:
\[
E_0(x) = \#\{(m,n) \in \Z^2 : 0\leq m,n \leq x,\ p\,|\,m \text{ and }p\,|\,n \text{ for some } p > \xi\},
\]
and
\[
E_i(x) = \#\{(m,n) \in \Z^2 : 0\leq m,n \leq x,\ p^2\,|\, f_i(m,n) \text{ for some } p > \xi\}.
\]
The proof of~\cite{GouveaMazur}*{Proposition~2}, essentially unchanged, shows that $E(x) := \sum_{i=0}^t E_i(x)$ gives an upper bound for the error term of our approximation, as follows.
\begin{proposition}
If $\xi > \max\{ \Delta(F), M\}$ then
\[
N'(x) - E(x) \leq N(x) \leq N'(x). \eqno\qed
\]
\end{proposition}

\noindent The proposition implies that
\[
N(x) = N'(x) + O(E(x)),
\]
which is why we think of $\xi$ as giving us the notion of ``small prime relative to $x$.'' The choice of $(1/3)\log x$ is somewhat flexible (see~\cite{GouveaMazur}*{\S4}); what is important is that when $\ell$ is a squarefree integer divisible only by primes \emph{smaller} than $\xi$ then
\begin{equation}
\label{eq:xi bound}
\ell \leq \prod_{p < \xi} p = \exp\Bigg(\sum_{p < \xi} \log p \Bigg) \leq e^{2\xi} = x^{2/3},
\end{equation}
where the last inequality follows from the estimate
\[
\sum_{p < \xi} \log p \le \sum_{p < \xi} \log \xi = \pi(\xi)\log\xi < 2\xi,
\]
and $\pi(x) = \#\{ p\text{ prime} : p < x\}$; see~\cite{Stopple}*{p.\ 105}.

In~\cite{Greaves}, Greaves shows that as $x \to \infty$
\[
E(x) = O\left(\frac{x^2}{(\log x)^{1/3}}\right).
\]
Greaves' proof requires the hypothesis that no irreducible factor of $F$ have degree greater than 6, which explains the presence of this hypothesis in Theorem~\ref{thm:ModifiedSieve}. Theorem~\ref{thm:ModifiedSieve} thus follows from the next lemma.

\begin{lemma}
With $C$ as in Theorem~\ref{thm:ModifiedSieve}, we have
\[
N'(x) = Cx^2 + O\left(\frac{x^2}{\log x}\right)\qquad\textup{as } x \to \infty.
\]
\end{lemma}

\begin{proof}
Let $\ell$ be a squarefree integer divisible only by primes smaller than $\xi$, and such that $\gcd(\ell,M) = 1$. Let $N_\ell(M,\calS;x)$ be the number of pairs of integers $(m,n)$ such that 
\[
0 \leq m,n \leq x,\quad  (m\bmod M, n\bmod M) \in \calS, \quad\text{and } F(m,n) \equiv 0 \bmod \ell^2.
\]
For a fixed congruence class modulo $\ell^2$ of solutions $F(m_0,n_0) \equiv 0 \bmod \ell^2$, satisfying $(m_0\bmod M, n_0\bmod M) \in \calS$, we count the number of representatives in the box $0\leq m,n, \leq x$, and obtain
\[
N_\ell(M,\calS;x) = \frac{x^2\cdot|\calS|}{M^2}\cdot\frac{\rho(\ell^2)}{\ell^4} + O\left(x\cdot\frac{\rho(\ell^2)}{\ell^{2}}\right),
\]
where the implied constant depends on $F, M$ and $\calS$, but not on $\ell$ or $x$. By the inclusion-exclusion principle we have
\[
N'(x) = \sum_\ell \mu(\ell)N_\ell(M,\calS;x),
\]
where $\mu$ denotes the usual M\"obius function and the sum runs over squarefree integers that are divisible only by primes smaller than $\xi$ and that are relatively prime to $M$. Thus, by~\eqref{eq:xi bound},
\begin{align*}
N'(x) &= \frac{x^2\cdot |\calS|}{M^2}\sum_\ell\mu(\ell) \frac{\rho(\ell^2)}{\ell^{4}} + O\Bigg(x\cdot 
         \sum_{\ell\leq x^{2/3}}\frac{\rho(\ell^2)}{\ell^{2}}\Bigg) \\
      &= \frac{x^2\cdot |\calS|}{M^2}\prod_{p < \xi,\ p\,\nmid\,M}\left(1 - \frac{\rho(p^2)}{p^{4}}\right) 
         + O\Bigg(x\cdot \sum_{\ell\leq x^{2/3}}\frac{\rho(\ell^2)}{\ell^{2}}\Bigg).
\end{align*}
Assume that $x$ is large enough so that $\xi > M$. Then, by Lemma~\ref{L:rho props}, we have
\begin{align*}
\prod_{p \geq \xi} \left(1 - \frac{\rho(p^2)}{p^{4}}\right) &= \prod_{p \geq \xi} \left(1 - O\left(\frac{1}
                                                               {p^{2}}\right)\right) = 1 - \sum_{p \geq 
                                                               \xi} O\left(\frac{1}{p^{2}}\right) \\
                                                            &= 1 - O\left(\int_{t \geq \xi} \frac{1}{t^2}\, 
                                                               dt\right) = 1 - O\left(\frac{1}{\xi}\right)
\end{align*}
Hence
\[
N'(x) = \frac{x^2\cdot |\calS|}{M^2}\prod_{p\,\nmid\,M}\left(1 - \frac{\rho(p^2)}{p^{4}}\right) + O\left(\frac{x^2}{\xi}\right) + O\Bigg(x\cdot \sum_{\ell\leq x^{2/3}}\frac{\rho(\ell^2)}{\ell^{2}}\Bigg).
\]
By Lemma~\ref{L:rho props}, we have 
\[
O\Bigg(x\cdot \sum_{\ell\leq x^{2/3}}\frac{\rho(\ell^2)}{\ell^{2}}\Bigg) = O\Bigg(x\cdot\sum_{\ell \leq x^{2/3}}d_k(\ell)\Bigg) = O(x\cdot x^{2/3}\log^{k-1} x),
\]
with $k = \deg F + 1$, where we have used the well-known fact that
\[
\sum_{n\leq x} d_k(n) = O(x \log^{k-1} x);
\]
see, for example, \cite{IwaniecKowalski}*{(1.80)}. Since $\xi = (1/3)\log x$, it follows that
\begin{align*}
N'(x) &= \frac{x^2\cdot |\calS|}{M^2}\prod_{p\,\nmid\,M}\left(1 - \frac{\rho(p^2)}{p^{4}}\right) + O
         \left(\frac{x^2}{\xi}\right) + O(x\cdot x^{2/3}\log^{k-1} x) \\
      &= \frac{x^2\cdot |\calS|}{M^2}\prod_{p\,\nmid\,M}\left(1 - \frac{\rho(p^2)}{p^{4}}\right) + O
         \left(\frac{x^2}{\log x}\right), 
\end{align*}
which concludes the proof.
\end{proof}

\subsection{Making sure that \texorpdfstring{$C$}{C} does not vanish}
\label{S:nonvanishing}

In this section we explore the possibility that the constant $C$ for the principal term of $N(x)$ is zero. This will depend on the particular binary form $F(m,n)$, the integer $M$ and the set $\calS$. For any prime $p \nmid M$, let 
\[
C_p = \left(1 - \frac{\rho(p^2)}{p^{4}}\right),
\]
so that $C = \frac{|\calS|}{M^2}\prod_{p\nmid M} C_p$. For $p \nmid M$ we know that $\rho(p^2) = O(p^2)$ (see Lemma~\ref{L:rho props}); hence $C$ vanishes if and only if either $\calS = \emptyset$, or one of the factors $C_p$ vanishes.

\begin{lemma}
With notation as above, if $p \nmid M$ and $p \geq \deg F$, then $C_p \neq 0$.
\label{L:Cnotzero}
\end{lemma}

\begin{proof} 
If $p\nmid M$ then $C_p = 0$ if and only if $\rho(p^2) = p^4$, which happens if and only if all pairs of integers $(m,n)$ modulo $\Z/p^2\Z$ are solutions to $F(m,n) \equiv 0 \bmod 
p^2$. But then \emph{all} pairs of integers $(m,n)$ give solutions to the given 
congruence equation.  This can happen only if $p < \deg(F)$; see the beginning of~\S\ref{ss: Main Results}.
\end{proof}

\subsection{An application of the modified sieve}

\begin{corollary}[Pseudo-squarefree sieve]
Let $F(m,n)\in \Z[m,n]$ be a homogeneous binary form of degree $d$. Assume that no square of a nonunit in $\Z[m,n]$ divides $F(m,n)$, and that no irreducible factor of $F$ has degree greater than $6$. Fix 
\begin{itemize}
\item a sequence $S = (p_1,\dots, p_r)$ of distinct prime numbers and 
\item a sequence $T = (t_1,\dots,t_r)$ of nonnegative integers.
\end{itemize}
Let $M$ be an integer such that $p^2\,|\,M$ for all primes $p < \deg F$ and $p_1^{t_1 + 1}\cdots p_r^{t_r + 1}\, |\, M$. Suppose that there exist integers $a, b$ such that
\begin{equation}
\label{eq: squarefree at small primes}
F(a,b) \not\equiv 0 \bmod p^2\quad\text{whenever } p\,|\, M \text{ and }p \neq p_i\text{ for any i},
\end{equation}
and such that
\begin{equation}
\label{eq: prescribed integer}
v_{p_i}(F(a,b)) = t_i\quad\text{for all $i$}. 
\end{equation}
Then there are infinitely many pairs of integers $(m,n)$ such that 
\begin{equation}
\label{eq: congruences for m and n}
m \equiv a \bmod M, \qquad n \equiv b \bmod M,
\end{equation}
and 
\[
F(m,n) = p_1^{t_1}\cdots p_r^{t_r}\cdot \ell,
\]
where $\ell$ is squarefree and $v_{p_i}(\ell) = 0$ for all $i$. 
\label{cor:ModifiedSieve}
\end{corollary}

\begin{proof}
Let $\calS = \{(a,b)\}$. By Theorem~\ref{thm:ModifiedSieve}, there are infinitely many pairs of integers $(m,n)$ such that
\[
m \equiv a \bmod M,\quad n \equiv b \bmod M, \quad \text{and }
F(m,n) \not\equiv 0 \bmod p^2\quad\text{whenever } p\nmid M,
\]
(Note that $|\calS| = 1$ and $C \neq 0$ by Lemma~\ref{L:Cnotzero}.) Condition~\eqref{eq: squarefree at small primes} then guarantees that $F(m,n)$ is not divisible by the square of any prime outside the sequence $S$. We also have
\[
m \equiv a \bmod p_i^{t_i + 1},\quad n \equiv b \bmod p_i^{t_i + 1}, \quad\text{for all }i,
\]
because $p_i^{t_i + 1} \,|\, M$ for all $i$, and hence
\[
F(m,n) = F(a,b) \bmod p_i^{t_i + 1}  \quad\text{for all }i.
\]
Using condition~\eqref{eq: prescribed integer}, we conclude that 
\[
v_{p_i}(F(m,n)) = t_i.
\eqno\qed
\]
\hideqed
\end{proof}

\section{Proof of Theorems~\ref{T: Main Theorem} and~\ref{T: Main Theorem II}}
\label{S:main thm}

For a finite extension $L/k$ of number fields, we let $S(L/k)$ denote the set of unramified prime ideals of $k$ that have a degree $1$ prime over $k$ in $L$. Given two sets $A$ and $B$, we write $A \doteq B$ if $A$ and $B$ differ by finitely many elements and $A \sqsubseteq B$ if $x \in A \implies x \in B$ with finitely many exceptions.

\begin{proposition}[Bauer, \cite{Neukirch}*{p.\ 548}]
\label{P: Bauer}
Let $k$ be a number field, $N/k$ a Galois extension of $k$ and $M/k$ an arbitrary finite extension of $k$.  Then
\begin{equation*}
S(M/k) \sqsubseteq S(N/k) \iff M \supseteq N.
\end{equation*}
\end{proposition}

\begin{lemma}
\label{L: needed prime}
Let $f(t) \in \Z[t]$ be an irreducible nonconstant polynomial, and let $N = \Q[t]/f(t)$. Let $\mu_3$ denote the group of third roots of unity, and suppose that $\Q(\mu_3) \nsubseteq N$. Then there are infinitely many rational primes $p$ such that $p \equiv 2 \pmod{3}$ and such that there exists a degree-$1$ prime $\mathfrak{p}\subseteq N$ lying over $p$.
\end{lemma}

\begin{proof}
Since $\F_p^\times$ contains an element of order $3$ if and only if $3|(p-1)$, it follows that
\[
S(\Q(\mu_3)/\Q) \doteq \{p \in \Z : p\text{ prime and } p \equiv 1 \bmod 3\}.
\]
Suppose that the following implication holds (with possibly finitely many exceptions):
\[
p \in \Z \text{ has a degree 1 prime in } N \implies p \equiv 1 \bmod 3.
\]
Then
\[
S(N/\Q) \sqsubseteq S(\Q(\mu_3)/\Q).
\]
It follows from Proposition~\ref{P: Bauer} that $\Q(\mu_3) \subseteq N$, a contradiction.
\end{proof}

A similar argument proves the following entirely analogous lemma.

\begin{lemma}
\label{L: needed prime II}
Let $g(t) \in \Z[t]$ be an irreducible nonconstant polynomial, and let $N = \Q[t]/g(t)$. Let $\mu_4$ denote the group of fourth roots of unity, and suppose that $\Q(\mu_4) \nsubseteq N$. Then there are infinitely rational primes $p$ such that $p \equiv 3 \pmod{4}$ and such that there exists a degree-$1$ prime $\mathfrak{p}\subseteq N$ lying over $p$.
\qed
\end{lemma}

\begin{proof}[Proof of Theorem~\ref{T: Main Theorem}]
Since the surface in $\PP_\Q(1,1,2,3)$ given by an equation of the form~\eqref{eq:sextic twists} is smooth (by the definition of a del Pezzo surface), it follows that $F_1$ is a squarefree binary form of degree $6$ (see \S\ref{S:blow-up}). Blowing up the anticanonical point $[0:0:1:1]$ of $X$ we obtain an elliptic surface $\rho\colon\mathcal{E} \to \PP^1_\Q$ whose fiber above $[m:n] \in \PP^1(\Q)$ is isomorphic to a curve in $\PP^2_\Q$ whose affine equation is given by
\begin{equation}
\label{eq:fiber}
y^2 = x^3 + F(m,n)
\end{equation}
This is an elliptic curve for almost all $[m:n]$.

Write $c = p_1^{\alpha_1}\cdots p_r^{\alpha_r}$, where the $p_i$ are distinct primes. Let $S = (p_1,\dots,p_r)$, $T = (0,\dots,0)$ and let 
\begin{equation*}
M = \left(2\cdot 3\cdot 5\right)^3\cdot (p_1\cdots p_r).
\end{equation*}
Since $F_1(m,n)$ has no fixed prime divisors, we know that for each prime $p \,|\, M$ with $p \neq p_i$ for all $i$ there exist congruence classes $a_p, b_p$ modulo $p^2$ such that
\begin{equation*}
F_1(a_p,b_p) \not\equiv 0 \bmod p^2.
\end{equation*}
Similarly, for a prime $p_i$ in the sequence $S$ there exist congruence classes $a_{p_i}, b_{p_i}$ modulo $p_i$ such that
\begin{equation*}
F_1(a_{p_i},b_{p_i}) \not\equiv 0 \bmod p_i;
\end{equation*}
in other words, $v_{p_i}(F_1(a_{p_i},b_{p_i})) = 0$.  By the Chinese remainder theorem there exist congruence classes $a, b$ modulo $M$ such that
\begin{equation}
(a,b) \equiv 
\begin{cases}
(a_p,b_p) \bmod p^2     & \text{for all primes $p$ such that $p\,|\,M$, $p \neq p_i$ for any $i$,} \\
(a_{p_i},b_{p_i}) \bmod p_i & \text{for all primes $p_i$ in the sequence $S$.}
\end{cases}
\label{eq:Congruences for a 1}
\end{equation}
By Corollary~\ref{cor:ModifiedSieve}, applied to $F_1,S,T,M,a$ and $b$ as above, there is an infinite set $\sF_1$ of pairs $(m,n) \in \Z^2$ such that 
\[
F_1(m,n) = \ell,
\]
where $\ell$ is a squarefree integer with $\gcd(c,\ell) = 1$, by our choice of $S$ and $T$. Note that the elements $m$, $n$ of each pair must be coprime since $F_1(m,n)$ is squarefree. Furthermore, the congruence class of $\ell$ modulo $2^3\cdot 3^3$ is fixed (by our choice of $M$) and nonzero (because $\ell$ is squarefree).  Thus, for $(m,n) \in \sF_1$ we have
\[
F(m,n) = c \ell\qquad \gcd(c,\ell) = 1,
\]
and the congruence class of $c \ell/2^{v_2(c\ell)}3^{v_3(c\ell)}$ modulo $2^2\cdot 3^2$ is fixed and nonzero. 

By Lemma~\ref{L: needed prime}, applied to a number field $N := \Q[t]/f_i(t,1)$ such that~\eqref{E:GaloisCondition} holds, there is a rational prime $q \equiv 2 \bmod 3$ and a degree $1$ prime $\mathfrak{q}$ in $N$ lying over $q$. In fact, we may choose $q$ so that $q > 5$, $\gcd(q,c) = 1$, and so that it does not divide the discriminant of $f_i(t,1)$.

We apply Corollary~\ref{cor:ModifiedSieve} again to $F_1(m,n)$. This time we let $S = (p_1,\dots,p_r,q)$ and $T = (0,\dots,0,2+6k)$, where $k$ is a large positive integer\footnote{We will pick $k$ large enough to ensure that $C \neq 0$ upon application of the pseudo-squarefree sieve.}. Let  
\[
M = (2\cdot 3 \cdot 5)^3\cdot(p_1\cdots p_r)\cdot q^{3+6k}.
\]
We claim that there exist integers $m_q, n_q$ such that
\[
v_q(F_1(m_q,n_q)) = 2 + 6k.
\]
Indeed, since $q$ has a prime $\mathfrak{q}$ of degree $1$ in $N$ and it does not divide the discriminant of $f_i(t,1)$, the equation
\[
f_i(t,1) = 0
\]
has a simple root in $\F_q$. By Hensel's lemma, this solution lifts to a root in $\Q_q$. Hence $F_1(t,1) = 0$ has a root in $\Q_q$. Approximating this solution by a rational number $r_q = m_q/n_q$ we can control $v_q(F_1(r_q,1))$ modulo $6$; i.e., there exists a pair $(m_q,n_q) \in \Z^2$ of coprime integers such that $v_q(F_1(m_q,n_q)) = 2 + 6k$ for some (possibly very large) positive integer $k$. By the Chinese remainder theorem, there exists a pair of integers $(a,b)$ simultaneously satisfying~\eqref{eq:Congruences for a 1} and
\begin{equation}
\label{eq:condition3}
a \equiv m_q \bmod q^{3+6k},\quad \text{and }b \equiv n_q \bmod q^{3+6k}.
\end{equation}
By Corollary~\ref{cor:ModifiedSieve}, applied to $F_1,S,T,M,a$ and $b$ as above, there is an infinite set $\sF_2$ of pairs $(m,n) \in \Z^2$ such that 
\[
F_1(m,n) = q^{2+6k}\eta,
\]
for some squarefree integer $\eta$ with $\gcd(c,q \eta) = \gcd(q, \eta) = 1$, by our choice of $S$ and $T$. Suppose that $(m,n) \in \sF_2$. Then
\[
F(m,n) = c q^{2+6k} \eta\qquad \gcd(c,\eta) = \gcd(q,c \eta) = 1.
\]
Furthermore, we claim that $\gcd(m,n) = 1$. To see this, note that since $\eta$ is squarefree and $F_1$ is homogeneous of degree $6$,  then $\gcd(m,n)$ is some power of $q$; by \eqref{eq: congruences for m and n}, \eqref{eq:condition3}, and because $\gcd(m_q,n_q) = 1$, this power of $q$ must be $1$. As before, the congruence class of $c  q^{2+6k}\eta/2^{v_2(c\eta)}3^{v_3(c\eta)} \bmod 2^2\cdot 3^2$ is fixed, nonzero, and equal to that of $F_1(m,n)$ for $(m,n) \in \sF_1$ (by our choice of $a$ and $b$).

Whenever~\eqref{eq:fiber} is smooth, we write $W(F(m,n))$ for its root number. By Corollary~\ref{cor:flipping}, if $(m_1,n_1) \in \sF_1$ and $(m_2,n_2) \in \sF_2$ then
\[
W(F(m_1,n_1)) = - W(F(m_2,n_2)).
\]
Zariski density of rational points on $X$ now follows by arguing as in Remark~\ref{R: reduction}.
\end{proof}

The proof of Theorem~\ref{T: Main Theorem II} is similar; we give enough details so that the interested reader can reconstruct it from the proof of Theorem~\ref{T: Main Theorem}.

\begin{proof}[Proof of Theorem~\ref{T: Main Theorem II}]
Blowing up the singular locus of $X$, as well as the base-point $[0:0:1:1]$ of $|{-K}_X|$, we obtain an elliptic surface $\rho\colon\mathcal{E} \to \PP^1_\Q$ whose fiber above $[m:n] \in \PP^1(\Q)$ is isomorphic to a curve in $\PP^2_\Q$ whose affine equation is given by
\begin{equation}
\label{eq:fiberII}
y^2 = x^3 + G(m,n)x,
\end{equation}
which is an elliptic curve for almost all $[m:n]$.

We apply Corollary~\ref{cor:ModifiedSieve} twice, as in the proof of Theorem~\ref{T: Main Theorem}. First, we apply it to $G_1(m,n)$ by taking $S = (p_1,\dots,p_r)$, $T = (0,\dots,0)$, where $c = p_1^{\alpha_1}\cdots p_r^{\alpha_r}$, and the $p_i$ are distinct primes. We use
\begin{equation*}
M = \left(2^2\cdot 3\right)^3\cdot (p_1\cdots p_r).
\end{equation*}
This way we obtain an infinite set $\sF_1$ of coprime pairs of integers $(m,n)$ such that
\[
G(m,n) = c \ell\qquad \gcd(c,\ell) = 1,
\]
and the congruence class of $c \ell/2^{v_2(c\ell)}3^{v_3(c\ell)}$ modulo $2^4\cdot 3^2$ is fixed and nonzero. 

By Lemma~\ref{L: needed prime II}, applied to a number field $N := \Q[t]/g_i(t,1)$ such that~\eqref{E:GaloisConditionII} holds, there is a rational prime $q \equiv 3 \bmod 4$ and a degree $1$ prime $\mathfrak{q}$ in $N$ lying over $q$. In fact, we may choose $q$ so that $q > 5$, $\gcd(q,c) = 1$, and so that it does not divide the discriminant of $g_i(t,1)$.

We apply Corollary~\ref{cor:ModifiedSieve} again to $G_1(m,n)$ with $S = (p_1,\dots,p_r,q)$ and $T = (0,\dots,0,2+4k)$, where $k$ is a large positive integer, and  
\[
M = (2^2 \cdot 3)^3\cdot(p_1\cdots p_r)\cdot q^{3+4k}
\]
Using Hensel's lemma as in the proof of Theorem~\ref{T: Main Theorem}, we obtain a different infinite set $\sF_2$ of coprime pairs integers $(m,n)$ such that
\[
G(m,n) = c q^{2+4k} \eta\qquad \gcd(c,\eta) = \gcd(q,c\eta) = 1,
\]
where $\eta$ is a squarefree integer. As before, the congruence class of $c q^{2+4k}\eta/2^{v_2(c\eta)}3^{v_3(c\eta)}$ modulo $2^4\cdot 3^2$ is fixed, nonzero, and equal to that of $G_1(m,n)$ for $(m,n) \in \sF_1$ (by our choice of $a$ and $b$).

Whenever~\eqref{eq:fiberII} is smooth, we write $W(G(m,n))$ for its root number. By Corollary~\ref{cor:flippingII}, if $(m_1,n_1) \in \sF_1$ and $(m_2,n_2) \in \sF_2$ then
\[
W(G(m_1,n_1)) = - W(G(m_2,n_2)).
\]
Zariski density of rational points on $X$ now follows by arguing as in Remark~\ref{R: reductionII}. 
\end{proof}

\section{Diagonal del Pezzo surfaces of degree \texorpdfstring{$1$}{1}}
\label{S:diagonal surfaces}

We begin this section with two examples of del Pezzo surfaces of degree $1$ that show how the sieving technique used in the proof of Theorems~\ref{T: Main Theorem} and~\ref{T: Main Theorem II} can fail. In one case, however, we can show that rational points are Zariski dense, by exhibiting explicit nontorsion sections of the associated elliptic surfaces.

\begin{example}
\label{E: example with sections}
Consider the del Pezzo surface of degree $1$ given by 
\[
w^2 = z^3 + 27x^6 + 16y^6
\]
in $\PP_\Q(1,1,2,3)$ Let $\rho\colon\calE\to \PP^1_\Q$ be its associated elliptic fibration. The elliptic curve $E_{m,n}$ above the point $[m:n] \in \PP^1(\Q)$ is given by
\[
E_{m,n}\colon\quad y^2 = x^3 + 27m^6 + 16n^6.
\]
We claim that $W(E_{m,n}) = +1$ \emph{for all} $[m:n] \in  \PP^1(\Q)$. We may assume that $\gcd(m,n) = 1$. Let $\alpha = 27m^6 + 16n^6$, and suppose that $p \geq 5$ divides $\alpha$ (in particular, $p \nmid m$). Then
\[
-3 \equiv (4n^3/3m^3)^2 \bmod p,
\]
and thus $\legendre{-3}{p} = 1$; hence the product over $p^2\,|\,\alpha$ in \eqref{eq:root numbers} is equal to $1$. Using the notation of Proposition~\ref{T: root numbers of sextics}, it remains to see that $R(\alpha) = -1$. Since $\gcd(m,n) = 1$, we have $v_2(\alpha) = 4$ or $0$, according to whether $2\,|\, m$ or not. In either case, using Lemma~\ref{L:local root numbers at 2 and 3}, we see that
\[
W_2(\alpha)\cdot \legendre{-1}{\alpha_2} = 1\quad\text{for all } \alpha.
\]
Similarly, $v_3(\alpha) = 0$ or $3$ according to whether $3\nmid n$ or not. By Lemma~\ref{L:local root numbers at 2 and 3} it also follows that
\[
W_3(\alpha)\cdot (-1)^{v_3(\alpha)} = -1\quad\text{for all }\alpha,
\]
and hence $R(\alpha) = -1$, as desired.

The flipping technique of Corollary~\ref{cor:flipping} thus cannot possibly work! Furthermore, assuming the parity conjecture, it follows that $E_{m,n}$ has even Mordell-Weil rank for all $[m:n] \in \PP^1(\Q)$. In fact, we claim that \emph{all but finitely many} fibers have even rank $\geq 2$. To see this note the family contains the points
$$
(-3m^2,4n^3)\qquad \text{and} \qquad \left(\frac{9m^4}{4n^2}, \frac{27m^6}{8n^3} + 4n^3\right).
$$
We can check that these points are independent on the fiber above $[m:n] = [1:1]$, and thus they are independent as points on the generic fiber of $\calE$. Then Silverman's Specialization Theorem~\cite{SilvermanII}*{Theorem~11.4} shows that the points are independent for all but finitely many pairs $(m,n)$. Hence, rational points are Zariski dense on the original del Pezzo surface\footnote{In fact, this surface is not minimal. The two non-torsion sections of $\calE \to \PP^1_{\Q}$ correspond to exceptional curves on $X$ that are defined over $\Q$. Contracting these curves gives a del Pezzo surface of degree $3$ with a rational point.  This surface is unirational by the Segre-Manin Theorem.}.
\end{example}


\begin{example}
\label{E: example without sections}
Consider the del Pezzo surface of degree $1$ given by 
\[
w^2 = z^3 + 6(27x^6 + y^6)
\]
in $\PP_\Q(1,1,2,3)$. The elliptic curve $E_{m,n}$ above a point $[m:n] \subseteq \PP^1(\Q)$ of the associated elliptic surface $\calE\to \PP^1_\Q$ is given by
\[
E_{m,n}\colon\quad y^2 = x^3 + 6(27m^6 + n^6).
\]
As in Example~\ref{E: example with sections} we can show that $W(E_{m,n}) = +1$ for all $[m:n] \in \PP^1(\Q)$. However, we cannot find readily available sections; Zariski density of rational points on this surface remains an open question. 
\end{example}

The key point behind both of examples above is that condition~\eqref{E:GaloisCondition} on the form $F_1(m,n)$ fails. The following lemma gives a necessary condition for the failure of~\eqref{E:GaloisCondition} to occur, and suggests how to find the above examples.

\begin{lemma}
\label{L:Galois condition}
Let $F_1(m,n) = Am^6 + Bn^6 \in \Z[m,n]$, and assume that $\gcd(A,B) = 1$. Write $F_1 = \prod_i f_i$, where the $f_i \in \Z[m,n]$ are irreducible homogeneous forms. Let $\mu_3$ denote the group of third roots of unity. Then
\begin{equation}
\mu_3 \subseteq \Q[t]/f_i(t,1)\text{ for all } i \implies 3A/B \text{ is a rational square.}
\end{equation}
\end{lemma}

\begin{proof}
The proof is an exercise in Galois theory.  We will prove the case where $F_1$ is irreducible to illustrate the method. Choose a sixth root $\xi$ of $-B/A$ and an isomorphism $\Q[t]/(At^6 + B) \xrightarrow{\sim} \Q(\xi)$. Suppose that $\Q(\mu_3)\subseteq \Q(\xi)$, so that $\Q(\xi)/\Q$ is a Galois extension of degree $6$.  Its unique quadratic subextension is $\Q(\mu_3) = \Q(\sqrt{-3})$, hence
\[
\xi^3 = a + b\sqrt{-3}\quad\text{for some }a, b \in \Q.
\]
Squaring both sides of the above equation and rearranging we obtain
\[
-B/A - a^2 + 3b^2 = 2ab\sqrt{-3}
\]
so that $ab = 0$. Since $\xi^3 \notin \Q$, it follows that $a = 0$ and $B/A = 3b^2$.
\end{proof}

If $3A/B$ is a rational square, it is often the case that not all fibers of the associated elliptic surface have positive root number: the $2$-adic and $3$-adic part of $Am^6 + Bn^6$ may vary enough to guarantee the existence of infinitely many fibers with root number $-1$. This idea, together with Theorem~\ref{T: Main Theorem}, are the necessary ingredients in the proof of Theorem~\ref{C:diagonal}.

\begin{proof}[Proof of Theorem~\ref{C:diagonal}]
Let $F(x,y) = Ax^6 + By^6$ and put $c = \gcd(A,B)$. Write $F_1(x,y) = A_1x^6 + B_1y^6$, where $cA_1 = A$ and $cB_1 = B$. One easily checks that $F_1$ has no fixed prime factors. Write $F_1 = \prod_i f_i$, where the $f_i \in \Z[x,y]$ are irreducible homogeneous forms. If $3A/B$ is not a rational square then it follows from Lemma~\ref{L:Galois condition} that 
\[
\mu_3 \nsubseteq \Q[t]/f_i(t,1)\qquad\text{ for some } i,
\]
so by Theorem~\ref{T: Main Theorem}, $X(\Q)$ is Zariski dense in $X$.

If, on the other hand, $3A/B$ is a rational square, then by assumption $c = 1$ and $9 \nmid AB$. After possibly interchanging $A$ and $B$, we may write $A = 3a^2$ and $B = b^2$ for some relatively prime $a, b \in \Z$ not divisible by $3$. A smooth fiber above $[m:n] \in \PP^1(\Q)$ of the elliptic surface $\calE \to \PP^1_\Q$ associated to $X$ is the plane curve
\[
E_{\alpha}\colon\quad y^2 = x^3 + \alpha,
\]
where $\alpha = 3a^2m^6 + b^2n^6$. Arguing as in Example~\ref{E: example with sections} we see that the product over $p^2\,|\,\alpha$ in~\eqref{eq:root numbers} is equal to $1$. 

To conclude the proof, it suffices to show that there are infinitely many pairs $(m,n)$ of relatively prime integers such that $R(\alpha) = 1$ (see Proposition~\ref{T: root numbers of sextics} for the definition of $R(\alpha)$). To construct such pairs $(m,n)$, first suppose that $3\mid n$ (whence $3\nmid m$). Then $v_3(\alpha) = 1$ and $\alpha_3 \equiv 1 \bmod 3$, so by Lemma~\ref{L:local root numbers at 2 and 3}
\[
W_3(\alpha)\cdot(-1)^{v_3(\alpha)} = (-1)\cdot(-1) = 1.
\]
Next, we compute the product 
\[
w_2 := W_2(\alpha)\legendre{-1}{\alpha_2}.
\]
We proceed by analyzing two cases, according to the $2$-adic valuation of $b$, which we may assume is either $0, 1$ or $2$. We use Lemma~\ref{L:local root numbers at 2 and 3} to compute the local root number at $2$:
\begin{enumerate}
\item $v_2(b) = 0$: choose $n$ even. Then, regardless of the value of $v_2(a)$ (which we may also assume is $0, 1$ or $2$), we obtain $v_2(\alpha)$ even and $\alpha_2 \equiv 3 \bmod 4$, whence $w_2 = 1$.
\item $v_2(b) = 1$ or $2$: choose $m$ odd, so that $v_2(\alpha) = 0$ and $\alpha_2 \equiv 3 \bmod 4$, whence $w_2 = 1$.
\end{enumerate}
In any case, there are infinitely many pairs $(m,n) \in \Z^2$ such that $R(3a^2m^6 + b^2n^6) = 1$, as desired.
\end{proof}

\begin{remark}
\label{R: no sections}
If $3A/B$ is a rational square, and either $\gcd(A,B) \neq 1$ or if $9\,|\, AB$ then it can happen that all the elliptic curves that are fibers of the rational surface associated to $X$ have root number $+1$ (see Examples~\ref{E: example with sections} and~\ref{E: example without sections}). Even when $9 \,|\, AB$ there are examples of surfaces, such as
\[
w^2 = z^3 + 3^5x^6 + 2^4y^6,
\]
where we were not able to find nontorsion sections.
\end{remark}

\section{Proof of Theorem~\ref{Thm: Weak-weak Approximation}}
\label{s: WWA}

We carry out the details for the case of a surface $X$ as in Theorem~\ref{T: Main Theorem}, the other case being similar. The fiber of $\rho$ above $[m:n] \in \PP^1(\Q)$ is isomorphic to the plane curve
\begin{equation}
\label{eq:fiberbis}
y^2 = x^3 + F(m,n)
\end{equation}
which is an elliptic curve for almost all $[m:n]$. As in Theorem~\ref{T: Main Theorem}, we write $c$ for the content of $F$ and $F_1(m,n) := (1/c)F(m,n)$. By Lemma~\ref{L: needed prime}, applied to a number field $N := \Q[t]/f_i(t,1)$ such that~\eqref{E:GaloisCondition} 
holds, there is a rational prime $q \equiv 2 \bmod 3$ and a prime $\mathfrak{q}$ in $N$ lying over $q$ of degree $1$ over $\Q$. We may assume that $q > 5$, $\gcd(c,q) = 1$, and that $q$ does not divide the discriminant of $f_i(t,1)$. Write $c = p_1^{\alpha_1}\cdots p_r^{\alpha_r}$, where the $p_i$ are distinct primes. Let $P_0 = \{2,3,5,p_1,\cdots,p_r,q,\infty\}$.

Fix a finite set of distinct primes $P = \{q_1\dots,q_s\}$ such that $P \cap P_0 = \emptyset$, as well as a point $[m_p:n_p] \in \PP^1(\Q_p)$ for each $p \in P$. We may assume that $m_p, n_p \in \Z_p$, and without loss of generality\footnote{In fact, we may only really assume that either $m_p \in \Z_p^\times$ or $n_p \in \Z_p^\times$. We can interchange the roles of $m_p$ and $n_p$ in any one step of the proof without much difficulty, so the assumption that $n_p \in \Z_p^\times$ is an artifact to clean up the details of the proof.} we will further assume that $n_p \in \Z_p^\times$ for every $p \in P$. Let $\epsilon > 0$ be given and choose an integer $N$ large so that 
\begin{equation}
\label{eq: N defined}
1/p^N < \epsilon\quad\text{and}\quad v_p(F_1(m_p,n_p)) < N \quad\text{for every $p \in P$.}
\end{equation}
Let
\[
S = (p_1,\dots,p_r,q_1,\dots,q_s),\quad T = \big(0,\dots,0,v_{q_1}(F_1(m_{q_1},n_{q_1})),\dots,v_{q_s}(F_1(m_{q_s},n_{q_s}))\big),
\]
and let 
\begin{equation*}
M = \left(2\cdot 3\cdot 5\right)^3\cdot (p_1\cdots p_r) \cdot (q_1\cdots q_s)^N.
\end{equation*}
Since $F_1(m,n)$ has no fixed prime factors, for any prime $p \,|\, M$ such that $p \neq p_i$ for all $i$ and $p \notin P$, there exist congruence classes $a_p, b_p$ modulo $p^2$ such that
\begin{equation*}
F_1(a_p,b_p) \not\equiv 0 \bmod p^2.
\end{equation*}
Similarly, for a prime $p_i$ with $1 \leq i \leq r$, there exist congruence classes $a_{p_i}, b_{p_i}$ modulo $p_i$ such that
\begin{equation*}
F_1(a_{p_i},b_{p_i}) \not\equiv 0 \bmod p_i.
\end{equation*}
By the Chinese remainder theorem there exist congruence classes $a, b$ modulo $M$ such that
\begin{equation}
(a,b)  \equiv
\begin{cases}
 (a_p,b_p) \bmod p^2     & \text{for primes $p$ such that $p\,|\,M$, $p \neq p_i$ for all $i$ and $p 
                                   \notin P$,} \\
(a_{p_i},b_{p_i}) \bmod p_i & \text{for primes $p_i$ with $1 \leq i \leq r$,} \\
(m_p,n_p) \bmod p^N     & \text{for primes $p \in P$.}
\end{cases}
\label{eq:Congruences for a}
\end{equation}
By construction,
\[
F_1(a,b) \equiv F_1(m_p,n_p) \bmod p^N \quad \text{ for all $p \in P$}.
\]
It follows from~\eqref{eq: N defined} that
\[
v_p(F_1(a,b)) = v_p(F_1(m_p,n_p))\quad\text{ for all $p \in P$}.
\]
By Corollary~\ref{cor:ModifiedSieve}, applied to $F_1,S,T,M,a,b$ as above, there is an infinite set $\sF_1$ of pairs $(m,n) \in \Z^2$ such that 
\[
F_1(m,n) = \ell,
\]
where $\ell$ is a squarefree integer with $\gcd(c,\ell) = 1$, by our choice of $S$ and $T$. Furthermore, the congruence class of $\ell$ modulo $2^3\cdot 3^3$ is fixed (by our choice of $M$) and nonzero (because $\ell$ is squarefree).  Thus, for $(m,n) \in \sF_1$ we have
\[
F(m,n) = c \ell\qquad \gcd(c,\ell) = 1,
\]
and the congruence class of $c \ell/2^{v_2(c\ell)}3^{v_3(c\ell)}$ modulo $2^2\cdot 3^2$ is fixed and nonzero. 

We apply Corollary~\ref{cor:ModifiedSieve} again to $F_1(m,n)$. This time we let 
\[
S = (p_1,\dots,p_r,q_1,\dots,q_s,q),\ \  T = \big(0,\dots,0,v_{q_1}(F_1(m_{q_1},n_{q_1})),\dots,v_{q_s}(F_1(m_{q_s},n_{q_s})),2+6k\big),
\]
where $k$ is a large positive integer (large enough to ensure that $C \neq 0$ upon application of the sieve), and we let
\begin{equation*}
M = \left(2\cdot 3\cdot 5\right)^3\cdot (p_1\cdots p_r) \cdot (q_1\cdots q_s)^N \cdot q^{3+6k}.
\end{equation*}
Arguing as in the proof of Theorem~\ref{T: Main Theorem}, using Hensel's lemma and Lemma~\ref{L: needed prime}, we can show that there exist integers $a_q, b_q$ such that
\[
v_q(F_1(a_q,b_q)) = 2 + 6k
\]
for some large positive integer $k$. By the Chinese remainder theorem, there exist congruence classes $a,b$ modulo $M$ such that~\eqref{eq:Congruences for a} holds, and in addition
\[
a \equiv a_q \bmod q^{3+6k},\quad \text{and } b \equiv b_q \bmod q^{3+6k}.
\]
By Corollary~\ref{cor:ModifiedSieve} there is an infinite set $\sF_2$ of pairs $(m,n) \in \Z^2$ such that 
\begin{equation}
\label{eq:fam2bis}
F_1(m,n) = q^{2+6k} \eta,
\end{equation}
where $\eta$ is a squarefree integer such that $\gcd(c,\eta) = \gcd(q,c \eta) = 1$ (by the choice of $S$ and $T$). In summary, for $(m,n) \in \sF_2$, we have
\[
F(m,n) = c q^{2+6k}\eta\qquad \gcd(c,\eta) = \gcd(q,c\eta) = 1,
\]
and the congruence class of $c q^{2+6k}\eta/2^{v_2(c \eta)} \bmod 2^2\cdot 3^2$ is fixed, nonzero, and equal to that of $F_1(m,n)$ for $(m,n) \in \sF_1$.

Whenever~\eqref{eq:fiberbis} is smooth, we write $W(F(m,n))$ for its root number. By Corollary~\ref{cor:flipping}, if $(m_1,n_1) \in \sF_1$ and $(m_2,n_2) \in \sF_2$, then
\[
W(F(m_1,n_1)) = - W(F(m_2,n_2)).
\]
Hence, there exists a pair $(m_0,n_0) \in \sF_1 \cup \sF_2$ such that $W(F(m_0,n_0)) = -1$. By the assumption that Tate-Shafarevich groups are finite we conclude that the fiber of $\rho$ above $[m_0:n_0]$ has positive Mordell-Weil rank, i.e., $[m_0:n_0] \in \calR$. By construction, $n_0 \neq 0$, and
\[
m_0 \equiv m_p \bmod p^N,\quad\text{and}\quad n_0 \equiv n_p \bmod p^N\quad\text{for all }p \in P.
\]
Hence
\[
\left|\frac{m_p}{n_p} - \frac{m_0}{n_0}\right|_p = |m_pn_0 - m_0n_p|_p \leq \frac{1}{p^N} < \epsilon\qquad\text{for all }p \in P,
\]
and $[m_0:n_0]$ is arbitrarily close to $[m_p:n_p]$ for all $p \in P$.  This concludes the proof of the theorem.
\qed


\end{document}